\documentclass[10.5pt,letterpaper,leqno]{article}
\usepackage[letterpaper, hmargin=2cm, vmargin={2cm, 2cm}]{geometry}
\usepackage{myDefs}
\usepackage[symbol]{footmisc}

\begin{document}
\title{Frequently oscillating families related to subharmonic functions}
\author{Adi Gl\"ucksam}
\maketitle
\begin{abstract}
The goal of this note is to extend the result bounding from bellow the minimal possible growth of frequently oscillating subharmonic functions to a larger class of functions that carry similar properties. We refine and find further applications for the technique presented by Jones and Makarov in their celebrated paper \cite{JoneMakarov1995}.
\end{abstract}
\section{Introduction}
In \cite{MyJM2020}, the author bounded from below the minimal possible growth of frequently oscillating subharmonic functions. We used a technique originally presented by Jones and Makarov in \cite{JoneMakarov1995}. Here we refine those techniques more and harness similar results for a larger class of functions by identifying the required properties to show the same lower bound. 

\subsection{A family of functions related to subharmonic functions}
Let $\psi:[0,1]\rightarrow\R_+$ be a monotone increasing continuous function, with $\limit t {0^+}\psi(t)=0$. {We will denote by $m_d$ the $d$-dimensional Lebesgue measure, and let $\mu$ be a measure of density $\omega$ for some non-negative function $\omega:\R^d\rightarrow[0,\infty)$, i.e. $d\mu(x):=\omega(x)dm_d(x)$.}

We define the collection $\mathcal M_\psi(\R^d)$ to be all Borel measures satisfying that for every ball, $B\subset\R^d$, and Borel-measurable set $E\subset B$
\begin{equation}\label{eq:measure}
\frac{\mu(E)}{\mu(B)}\le \psi\bb{\frac{m_d(E)}{m_d(B)}}.
\end{equation}
One example of such measure is defined by the density function $\omega(x):=\norm x{}^{\alpha}$ for $\alpha>0$ with the function $\psi(t)=4^d\cdot 6^\alpha\cdot t$. In fact, if a measure is not absolutely continuous with respect to Lebesgue's measure, then it cannot satisfy Condition (\ref{eq:measure}) since we can find a set $E$ so that the left hand side of the inequality is positive, while the right hand side is zero. We {may therefore assume without loss of generality that $\mu$ is defined as above}, making Condition (\ref{eq:measure}) well defined. However, not every measure, which is absolutely continuous with respect to Lebesgue's measure, satisfies (\ref{eq:measure}) for some function $\psi$. For example, define the function $\omega:\R^2\rightarrow\R^2$ by
$$
\omega(x)=\sumit k 1 \infty \indic{B\bb{2\pi k,\frac{\pi}{2k}}}(x)\cdot\cos\bb{k\cdot \abs {x-2\pi k}}+e^{-\abs x},
$$
and consider the measure induced by the function $\omega$, i.e.,
$$
\mu(E):=\integrate E{}{\omega(x)}{m_2(x)}.
$$
This measure is absolutely continuous with respect to Lebesgue's measure, while, for every $k\in\N$
\begin{align*}
\frac{\mu\bb{B\bb{2\pi k,\frac \pi{2k}}}}{\mu\bb{B\bb{2\pi k,1}}}&=\frac{\integrate{B\bb{2\pi k,\frac \pi{2k}}}{}{\cos\bb{k\cdot \abs {x-2\pi k}}+e^{-\abs x}}{m_2(x)}}{\integrate{B\bb{2\pi k,\frac \pi{2k}}}{}{\cos\bb{k\cdot \abs {x-2\pi k}}}{m_2(x)}+\integrate{B\bb{2\pi k,1}}{}{e^{-\abs x}}{m_2(x)}}\\
&\ge \frac{\pi-2}{\pi-2+\pi\cdot e^{-\pi k}}=1-\Theta(e^{-\pi k})\gg \frac{\pi^2}{4k^2}=\frac{m_2\bb{B\bb{2\pi k,\frac \pi{2k}}}}{m_2\bb{B\bb{2\pi k,1}}}\rightarrow 0 \text{ as } k\rightarrow \infty.
\end{align*}
We see that no matter which $\psi$ we choose, Condition (\ref{eq:measure}) would never hold, and so the class of measures for which Condition (\ref{eq:measure}) holds is a subset of the class of measures which are absolutely continuous with respect to Lebesgue's measure.

Given constants $A,B\ge 1$, and a measure $\mu\in\mathcal M_\psi(\R^d)$, we denote by $\mathcal F(A,B,\mu)$ the collection of upper semi-continuous functions $u:\R^d\rightarrow\R$ so that
\begin{itemize}
\item $u$ satisfies a weak maximum principle: for every $K\subset\R^d$ convex
\begin{equation}\label{eq:max_pric}
\underset {x\in K}\sup\; u(x)\le A\cdot \underset{x\in\partial K}\sup\; u(x).
\end{equation}
\item $u$ satisfies a weighted mean-value inequality: For every $r>r_0$,
\begin{equation}\label{eq:mean_val}
u(x)=ess\limitsup y x u(y)\le B\cdot\frac1{\mu(B(x,r))}\underset {B(x,r)}\int u d\mu.
\end{equation}
\end{itemize}
\begin{rmk}
We only require the maximum principle to hold in convex sets, because we only use it in cubes anyways, while it is easier to verify on convex sets.
\end{rmk}
\subsubsection {Basic Properties of the set $\mathcal F(A,B,\mu)$:} The following properties are simple to verify-
\begin{enumerate}[label=$\bullet$]
\item For every $A,B\ge 1$ and every measure $\mu$, $\mathcal F(A,B,\mu)$ contains all the non-negative constants. Note that if $A>1$, then it cannot contain negative constants at all.
\item For every $u\in \mathcal F(A,B,\mu)$ and constants $\lambda,\tau>0$ we have $\lambda\cdot u+\tau\in\mathcal F(A,B,\mu)$.
\item $\mathcal F(A,B,\mu)$ is closed under maximum operation, i.e., for every $u,v\in\mathcal F(A,B,\mu)$, $\max\bset{u,v}\in \mathcal F(A,B,\mu)$.
\end{enumerate}

\subsubsection{Motivating Examples}\label{subsubsec:examples}
\begin{enumerate}[label=(\roman*)]
\item Note that if $\mu$ is Lebesgue's measure, and $A=B=1$, then subharmonic functions all belong to $\mathcal F(1,1,m_d)$.
\item A second large class of examples arise from the following observation:
\begin{obs}\label{obs:Harnack}
Let $u:\R^d\rightarrow\R_+$ be a non-negative function satisfying Harnack's inequality, i.e., there exists a constant $c$ so that for every ball $B$
$$
\underset{B}\sup\; u\le c\cdot\underset{B}\inf\; u.
$$
Then $u\in\mathcal F(c,c,m_d)$.
\end{obs}
We include the proof here for completeness:
\begin{proof}
$\bullet$ To see that $u$ satisfies the weak maximum principle, let $K$ be a convex set and let $B$ be a ball large enough, which compactly contains the set $K$. Then following Harnack's inequality,
$$
\underset{x\in K}\sup\; u(x)\le \underset{x\in B}\sup\; u(x)\le c\cdot \underset{x\in B}\inf\; u(x)\le c\underset{x\in\partial K}\sup u(x).
$$
$\bullet$ To see that $u$ satisfies a weighted mean value theorem with Lebesgue's measure, note that for every ball $B$ containing $x$, 
$$
u(x)\le\underset{y\in B}\sup\; u(y)\le c\cdot{\underset{y\in B}\inf\; u(y)}=\frac c{m_d(B)}\integrate B{}{\underset{y\in B}\inf\; u(y)}{m_d}\le \frac c{m_d(B)}\integrate B{}{u(y)}{m_d}.
$$
\end{proof}
\item Another large class of examples of such functions are sub-solutions of quasilinear elliptic equations (called $A$-subharmonic functions).
\end{enumerate}
These and other examples will be discussed in Section \ref{sec:examples} more extensively.

\subsection{Frequently oscillating functions}
We next consider a sub-collection of functions in the family $\mathcal F(A,B,\mu)$ which satisfy additional requirements.

A cube $I$ is called a {\it basic cube} if $I=\prodit j 1 d [n_j,n_j+1)$ for $n_1,\cdots,n_d\in\Z$, i.e., $I$ is a half open half closed cube with edge length one, whose vertices lie on the lattice $\Z^d$. Let $N\in \N_{even}, N\gg 1$ and let $Q=\left[-\frac N2,\frac N2\right]^d$.

Given a function $u\in\mathcal F(A,B,\mu)$ and a basic cube $I$, we consider the properties: 
\begin{center}
\begin{minipage}[b]{.3\textwidth}
\vspace{-\baselineskip}
\begin{equation*}
\tag{P1}\underset{x\in I}\sup\; u(x)\ge 1 \label{eq:max_el}
\end{equation*}
\end{minipage}%
\hspace{0.2\textwidth}
\begin{minipage}[b]{.4\textwidth}
\vspace{-\baselineskip}
\begin{equation*}
\tag{P2} \frac{\mu\bb{I\cap Z_u}}{\mu(I)}\ge1-\Delta \label{eq:zero_el}
\end{equation*}
\end{minipage}
\end{center}
where $\Delta=\Delta(A,B,d,\mu)$ is a small enough constant, and $Z_u:= \bset{u\le 0}$. If a basic cube $I$ satisfies both properties, we say that {\it the function $u$ oscillates in $I$}. Otherwise, we say $I$ is a {\it rogue basic cube}.

Given a monotone non-decreasing function $f\,(t)\le t^d$, a function $u\in\mathcal F(A,B,\mu)$, and a cube $Q\subset\R^d$ with edge length $\ell(Q)$, we let
$$
\gamma_f^u(Q)=\gamma(Q)=\frac{\#\bset{I\subset Q ,I\text{ is a rogue basic cube}}}{f(\ell(Q))}.
$$
We say $u$ is {\it $f$-oscillating in $Q$} if $u$ is defined in a neighbourhood of $Q$, and $\gamma(Q)<1$. We say $u$ is {\it $f$-oscillating} if
$$
\limitsup N\infty\gamma([-N,N]^d)<1.
$$

{The study of estimates on the growth rate of subharmonic functions with a given well distributed zero set is a fundamental question dating back to the works of Jensen, Nevalinna, Phragm\'en and Lindel\"of, and others. However, frequently oscillating functions also have the requirement of having some uniform lower bound near its zero set. This special combination arose in \cite {Us2017}, while investigating the lower bound in Theorem 1A. In this paper we studied the minimal possible growth of entire functions in the support of translation invariant probability measures. Given such a measure, it follows from the point-wise ergodic theorem, that almost every entire function must satisfy that the subharmonic function $u_f(z):=\log\abs{f(z)}$ satisfies conditions $(P_1)$ and $(P_2)$ above when the number of rogue basic cube is proportional to Lebesgue's measure and the basic cube $I$ is replaced by some cube of fixed edge-length. Pursuing a more accurate lower bound and extending this for different asymptotic of the number of `rogue' basic cubes inspired \cite{MyJM2020}. The work presented here is a further extension to a more general setup.}

\subsection{The result and a word about the proof}
In this note we prove a lower bound on the minimal possible growth of frequently oscillating functions in $\mathcal F(A,B,\mu)$, measured by the asymptotic behavior of the function $R\mapsto M_u(R):=\underset{z\in B(0,R)}\sup\; u(z)$:
\begin{thm}\label{thm:quasi-lin-subsol}
Every $f$-oscillating function $u\in\mathcal F(A,B,\mu)$ must satisfy
$$
\limitinf R\infty \frac{\log(M_u(R))}{\frac{R}{1+\bb{\frac{f(R)}{R}}^{\frac1{d-1}}}\log^{\frac d{d-1}}\bb{2+\frac{f(R)}{R}}}>0,
$$
provided that $f(t)\le  c_0\cdot t^d$ for all $t$ large, where $c_0$ is a {small} constant, which depends on $A,B,d$ and the function $\psi$ {appearing in Condition \eqref{eq:measure}.}
\end{thm}

{We would have loved to provide a construction showing this result is optimal for at least some subclass which is not subharmonic functions. The most natural candidate is, off course, $\mathcal A$-subharmonic functions. However, the construction for subharmonic functions heavily relies on the connection between subharmonic functions and harmonic measures. While there have been numerous attempts to extend these for the more general setup of $p$-harmonic measures, these attempts were not successful. If you aim at a measure that recovers a $p$-harmonic function, which is what we need, the measure you get is not liner, i.e., not a measure (see e.g. \cite{Llorente2005}). A completely new idea for constructing optimal examples is therefore required, which, after consulting with experts in the field, seemed extremely challenging.}

\subsubsection{The idea of the proof}
The idea of the proof of Theorem \ref{thm:quasi-lin-subsol} is to use a similar stopping time argument, as the one appearing in Jones and Makarov's paper, \cite{JoneMakarov1995}. However some small modifications were required for the proof to work, for example due to the presence {of} the constants $A,\; B$ in the weak maximum principle (Condition (\ref{eq:max_pric})) and in mean-value property (Condition (\ref{eq:mean_val})). This method will be wrapped in a separate technical lemma (see Lemma \ref{lem:combinatorial} in Section \ref{sec:proof} below), and then used as a `black box' to prove Theorem \ref{thm:quasi-lin-subsol}. The author believes this lemma or a slight modification of it, can be used in other stopping time arguments as well.

In the first section of this paper we provide notation, and state a technical lemma, which is the heart of the stopping time argument used here. We then use the lemma to prove Theorem \ref{thm:quasi-lin-subsol}. In the second section we provide a full proof of the lemma. In the last section we provide examples of classes of functions contained in $\mathcal F(A,B,\mu)$ for some function $\psi$ and constants $A,\; B$.

The structure of the paper is as follows: In the first section of this paper we provide notation, and state the technical lemma, Lemma \ref{lem:combinatorial}, which is the heart of the stopping time argument used here. We then use the lemma to prove Theorem \ref{thm:quasi-lin-subsol}. In the second section we present the proof of Lemma \ref{lem:combinatorial}. In the last section we provide applications of Theorem \ref{thm:quasi-lin-subsol}. We consider the examples presented in \ref{subsubsec:examples}, show they belong to $\mathcal F(A,B,\mu)$ for some function $\psi$ and constants $A,\; B$, and present new examples of classes of functions contained in $\mathcal F(A,B,\mu)$.

\subsection{Acknowledgements}
The author is grateful to Ilia Binder, Eero Saksman, and Eugenia Malinnikova for several very helpful discussions. In particular, I would like to extend my gratitude to Mikhail Sodin who introduced me to Moser's work on Harnack inequalities, and to Pekka Pankka who asked the original question of minimal possible growth for frequently oscillating $A$-harmonic functions, which led to the more generalised result presented here. Lastly, the Author would like to thank the referee for all of their suggestions and useful remarks. Those truly made this paper better.

This material is based upon work supported by the National Science Foundation under Grant No. 1440140, while the author was in residence at the Mathematical Sciences Research Institute in Berkeley, California, during the spring semester of 2022.

\section{The proof}\label{sec:proof}
As mentioned in the introduction, the proof relies on a technical lemma. We begin by presenting some notation and stating this lemma. {Define the collection $\mathcal P:=\bset{I, \text{ basic cube }, I\subset Q}$, and let $\mathcal Q$ denote the collection of all possible cubes composed of elements of $\mathcal P$.}
\begin{figure}[htp]
\begin{minipage}[c]{.45\textwidth}
For every basic cube $I\in\mathcal P$ we denote by $Q_j(I)$ the cube centred of $I$ composed of $j$ layers of basic cubes surrounding $I$. Formally, we define this by recursion: $Q_0(I)=I$ and $Q_{j+1}(I)=\underset{J\in\mathcal P\atop \overline J\cap\overline Q_j(I)\neq\emptyset}\bigcup J$ (see Figure \ref{fig:Qs} to the right). Next, for every $x\in Q$ we let $I(x)$ be the basic cube so that $x\in I(x)$. Lastly, we write $A\lesssim B$ if there exists a constant $\alpha>0$ so that $A\le \alpha\cdot B$ and write $A\sim B$ whenever $A\lesssim B$ and $B\lesssim A$.
\end{minipage}\hfill
\begin{minipage}[c]{.5\textwidth}
\centering
\includegraphics[scale=0.85]{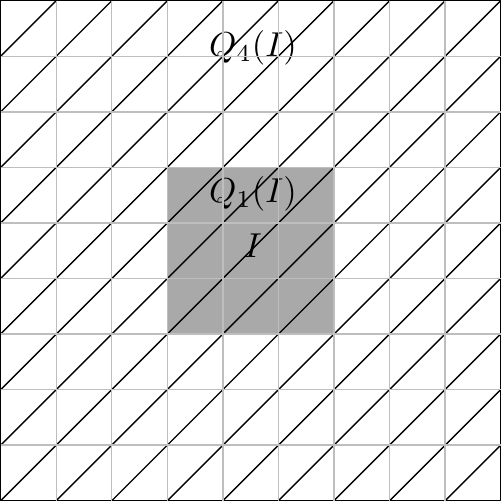}
\caption{$Q_1(I)$ is the gray cube centred at $I$, while $Q_4(I)$ is the stripped area. The grey lines are the boundary of the basic cubes.}
\label{fig:Qs}
\end{minipage}
\end{figure}

\begin{lem}\label{lem:combinatorial}
Given $\eps>0$, $r_0>1$, there exist constants $c_0\in\bb{0,1}$ and $N_0\in\N$ both depend on the dimension alone, so that for every $N>N_0$ and every collection of cubes, $\mathcal E\subset\mathcal P$, if $\frac{\# \mathcal E}{\#\mathcal P}\le \eps\cdot c_0$, then at least half of the basic cubes $I\in\mathcal P$ satisfy that there exists a sequence $\bset{k_j(I)}$ so that
\begin{enumerate}[label=\it{\arabic*.}]
\item \label{prop:gem} For every $x\in\partial Q_{k_j(I)}$ there exists $j=j(x)\ge\frac{r_0}2$ satisfying
\begin{enumerate}
\item $Q_{j(x)}(I(x))\subset Q_{k_{j+1}(I)}\setminus Q_{k_{j-1}(I)}$.\label{eq:contained}
\item $\#\bset{I\in\mathcal E, I\subset Q_{j(x)}(I(x))}\le \eps\cdot m_d\bb{Q_{j(x)}(I(x))}$. \label{eq:little_bad}
\end{enumerate}
\item $\#\bset{k_j(I)}\gtrsim \frac{N}{1+\bb{\frac{\#\mathcal E}{\eps\cdot N}}^{\frac1{d-1}}}\log^{\frac d{d-1}}\bb{2+\frac{\#\mathcal E}{\eps\cdot N}}$.\label{eq:sequence_length}
\end{enumerate}
\end{lem}

Let us first show how the proof of Theorem \ref{thm:quasi-lin-subsol} follows from this lemma, and then prove it:

\begin{proof}[The proof of Theorem \ref{thm:quasi-lin-subsol}]
{In general, the proof follows the same lines as the arguments done in \cite[p.8-9]{MyJM2020}. However, because we have a weak maximum principle as in \eqref{eq:max_pric}, and a generalized mean-value property as in \eqref{eq:mean_val}, some adjustments to the proof need to be made. We will indicate these adjustments throughout the proof.}

Let $c_0$ be the constant defined in the Lemma \ref{lem:combinatorial} with $\eps$ that will be chosen later and will depend on the $A,\;B,\;d$ and $\psi$. Let $f$ be a monotone non-decreasing function satisfying that $f(t)\le \eps\cdot c_0t^d$ for all large enough $t$, and let $u\in\mathcal F(A,B,\mu)$ be an $f$-oscillating function. We define the collection
\begin{eqnarray*}
\mathcal E&=&\bset{I\in \mathcal P, I\text{ is a rogue basic cube}}\\
\# \mathcal E&\le&\#\bset{I\in\mathcal P,I\text{ is a rogue basic cube}}=f(N)\cdot\gamma\bb{Q}<f(N)\le \eps\cdot c_0N^d.
\end{eqnarray*}
We may therefore apply Lemma \ref{lem:combinatorial} to conclude that for at least half of the basic cubes in $Q$ there exists a sequence $\bset{k_j(I)}$ so that 
\begin{enumerate}
\item For every $x\in\partial Q_{k_j(I)}$ there exists $r_x=\frac{j(x)}2\ge \frac{r_0}2$ so that
\begin{enumerate}
\item $B(x,2r_x)\subset Q_{j(x)}(I(x))\subset Q_{k_{j+1}(I)}\setminus Q_{k_{j-1}(I)}$.
\item $\#\bset{I\in\mathcal E, I\cap B\bb{x,r_x}\neq\emptyset}\le \#\bset{I\in\mathcal E, I\subset  Q_{j(x)}(I(x))}\le \eps\cdot m_d\bb{ Q_{j(x)}(I(x))}$.
\end{enumerate}
\item $\#\bset{k_j(I)}\gtrsim \frac{N}{1+\bb{\frac{\#\mathcal E}{\eps\cdot N}}^{\frac1{d-1}}}\log^{\frac d{d-1}}\bb{2+\frac{\#\mathcal E}{\eps\cdot N}}$.
\end{enumerate}
Following the {weak maximum principle in} \eqref{eq:max_pric}, there exists $x_j\in\partial Q_{k_j}$ so that $u(x_j)\ge \frac 1A\cdot M_u(Q_{k_j})$. On the other hand, following property 1(b) of the sequence $\bset{k_j(I)}$, there exists $r_j$ so that 
$$
\#\bset{I\in\mathcal E, I\cap B\bb{x_j,r_j}\neq\emptyset}\le \eps\cdot (2r_j+1)^d\le c_d\cdot\eps\cdot m_d(B(x_j,r_j)),
$$
for a constant $c_d$ which depends on the dimension alone. 

{The first adjustment to the proof is here. We need to bound from above the ratio $\frac{\mu\bb{B(x_j,r_j)\cap\bset{u>0}}}{\mu\bb{B(x_j,r_j)}}$. However, the measure $\mu$ is not, in general, translation invariant, and does not scale like Lebesgue's measure, and therefore we need a new argument to bound this quantity.} Using Condition (\ref{eq:measure}), we see that since the cubes are disjoint
\begin{align*}
&\frac{\mu\bb{B(x_j,r_j)\cap\bset{u>0}}}{\mu\bb{B(x_j,r_j)}}\le\frac{\mu\bb{B(x_j,r_j)\cap\underset{I\in\mathcal E}\bigcup I}}{\mu\bb{B(x_j,r_j)}}+\frac{\mu\bb{B(x_j,r_j)\setminus B(x_j,r_j-\sqrt d)}}{\mu\bb{B(x_j,r_j)}}+\frac{\underset{I\in\mathcal P\setminus\mathcal E\atop I\subset B(x_j,r_j)}\sum \mu \bb{I\cap\bset{u>0}}}{\mu\bb{B(x_j,r_j)}}\\
&\le \psi\bb{\frac{m_d\bb{B(x_j,r_j)\cap\underset{I\in\mathcal E}\bigcup I}}{m_d\bb{B(x_j,r_j)}}}+\psi\bb{\frac{m_d\bb{B(x_j,r_j)\setminus B(x_j,r_j-\sqrt d)}}{m_d\bb{B(x_j,r_j)}}}+\Delta\underset{I\in\mathcal P\setminus\mathcal E\atop I\subset B(x_j,r_j)}\sum\frac{\mu\bb{I}}{\mu\bb{B(x_j,r_j)}}\\
&\le \psi\bb{c_d\cdot\eps}+\psi\bb{\frac{d^{\frac d2}\cdot 2^d}{r}}+\Delta\le  \psi\bb{c_d\cdot\eps}+\psi\bb{\frac{d^{\frac d2}\cdot 2^d}{r_0}}+\Delta
\end{align*}
where the second inequality is since $I\nin\mathcal E$ satisfy property $P_2$. Note that this bound is uniform in $x_j$, and only depends on the number of `rogue' cubes intersecting the ball $B(x_j,r_j)$ and on the choice of $\Delta$, $\eps$, and $r_0$ which will depend on the function $\psi$, the constants $A$ and $B$, and the dimension $d$.

Next, following the weighted mean-value property {in \eqref{eq:mean_val}}, for every ball $B(x_j,r_j)$ we have
\begin{align*}
A\cdot u(x_j)&\le A\cdot B\cdot\frac1{\mu(B(x_j,r_j))}\underset {B(x_j,r_j)}\int u d\mu\le A\cdot B\cdot M_u(B(x_j,r_j))\cdot \frac{\mu(B(x_j,r_j)\cap\bset{u>0})}{\mu(B)}\\
& <A\cdot B\cdot M_u(B(x_j,r_j))\bb{  \psi\bb{c_d\cdot\eps}+\psi\bb{\frac{d^{\frac d2}\cdot 2^d}{r_0}}+\Delta}.
\end{align*}
{Note that, generalizing the original proof, an additional constant $A$ was required due to the weak maximum principle in \eqref{eq:max_pric}.} We conclude that if $\Delta$ and $\eps$ {are} chosen small enough (depending on $A,\; B$, and the function $\psi$) and $r_0$ is chosen large enough (depending on $\psi, d$ and the constants $A,B$) then there exists $\delta>0$ so that
$$
A\cdot B\cdot \bb{  \psi\bb{c_d\cdot\eps}+\psi\bb{\frac{d^{\frac d2}\cdot 2^d}{r_0}}+\Delta}<1-\delta\le e^{-\delta},
$$
implying that
$$
M_u(Q_{k_{j+1}})\ge M_u(B)\ge A\cdot u(x_j)\cdot e^\delta\ge M_u(Q_{k_j})\cdot e^\delta,
$$
following property 1(a) of the sequence $\bset{k_j(I)}$ and the maximum principle. Applying this argument inductively and using the maximum principle,
$$
\frac{M_u(I)}{M_u(Q)}=\prodit j 1 {\#\bset{k_j(I)}-1} \frac{M_u(Q_{k_j(I)})}{M_u(Q_{k_{j+1}(I)})}\le \exp\bb{-\#\bset{k_j(I)}\cdot\delta}\le \exp\bb{-c\cdot \frac{N}{1+\bb{\frac{\#\mathcal E}{\eps\cdot N}}^{\frac1{d-1}}}\log^{\frac d{d-1}}\bb{2+\frac{\#\mathcal E}{\eps\cdot N}}}
$$
for some constant $c$, which depends on the dimension, the constants $A,\; B$ and the function $\psi$. To get a similar bound  with $f(N)$ instead of $\#\mathcal E$, we note that the function $\psi$ defined by
$$
\psi(x):=\frac1{1+x^{\frac1{d-1}}}\cdot\log^{\frac d{d-1}}\bb{2+x}
$$
is monotone decreasing in some neighbourhood of $\infty$, which depends on the dimension. Combining this with the fact that $\#\mathcal E\le f(N)$, we see that for every $N$ large enough
\begin{eqnarray*}
\frac{M_u(I)}{M_u\bb{Q}}&\le&\exp\bb{-c\cdot\frac{N}{1+\bb{\frac{\#\mathcal E}{\eps\cdot N}}^{\frac1{d-1}}}\cdot\log^{\frac d{d-1}}\bb{2+\frac{\#\mathcal E}{\eps\cdot N}}}=\exp\bb{-c\cdot N\cdot \psi\bb{\frac{\#\mathcal E}{\eps\cdot N}}}\\
&\le&\exp\bb{-c\cdot N\cdot \psi\bb{\frac{f(N)}{\eps\cdot N}}}=\exp\bb{-c\cdot\frac{N}{1+\bb{\frac{f(N)}{\eps\cdot N}}^{\frac1{d-1}}}\cdot\log^{\frac d{d-1}}\bb{2+\frac{f(N)}{\eps\cdot N}}}.
\end{eqnarray*}
On the other hand, as without loss of generality $c_0<\frac12$, more than half of the basic cubes in $Q$ satisfy that $M_u(I)\ge 1$. We conclude that there exists at least one cube satisfying $M_u(I)\ge 1$ and the inequality above, implying that
\begin{eqnarray*}
M_u\bb{Q}&\ge&\exp\bb{c\cdot\frac N{1+\bb{\frac{f(N)}{\eps\cdot N}}^{\frac1{d-1}}}\log^{\frac d{d-1}}\bb{2+\frac{f(N)}{\eps\cdot N}}}M_u(I)\\
&\ge&\exp\bb{c\cdot\frac N{1+\bb{\frac{f(N)}{\eps\cdot N}}^{\frac1{d-1}}}\log^{\frac d{d-1}}\bb{2+\frac{f(N)}{\eps\cdot N}}}.
\end{eqnarray*}
As this holds for every $N\in\N$ large enough,
$$
\limitinf R\infty\frac{\log\bb{M_u(R)}}{\frac{R}{1+\bb{\frac{f(R)}{\eps\cdot R}}^{\frac1{d-1}}}\cdot\log^{\frac d{d-1}}\bb{2+\frac{f(R)}{\eps\cdot R}}}>0
$$
concluding our proof keeping in mind that $\eps$ was a constant which depends on the parameters $A,\; B,\; d$, and $\psi$.
\end{proof}
\section{The proof of Lemma \ref{lem:combinatorial}}
{This lemma is a refinement of \cite[Lemma 2.1]{MyJM2020}. In \cite[Lemma 2.1]{MyJM2020} we obtain a lower bound on the maximum of a subharmonic function which is frequently oscillating in some large cube. Here we extract the geometric property that allows us to find that lower bound, namely Property \eqref{prop:gem} in the Lemma, by using a collection of `rogue' cubes denoted $\mathcal E$ and bounding its density in all cubic annuli centered on some special sequence of cubic annuli, $\partial Q_{k_j}(I)$. This geometric property may find applications elsewhere as well.}

{There is also some novelty in this version of the lemma. In \cite[Lemma 2.1]{MyJM2020} we look at the density in cubes whose sizes ranges from 1 to $\infty$ and the density was bounded from bellow by some fixed constant corresponding to $1-\delta_d$ where $\delta_d$ is assumed to be small, i.e., we allow many `rogue' cubes. Here we look at the density in cubes of edge size at least $r_0$ and the density is bounded from above by $\eps$, both are parameters the user is free to choose. The proof is very similar to the proof done in \cite[Lemma 2.1]{MyJM2020} while keeping the parameters $r_0,\eps$ in mind. We include it here for completeness.}

Let $\rho:\mathcal P\rightarrow\bset{0,\cdots, N}$ be the function describing the edge-length of the smallest cube so that $Q_{\rho(I)}(I)$ does not contain many `rogue` cubes. {Here comes the first difference in the proof- we look at the density of the set of `rogue' cubes instead of the set where we have `good' behavior. In addition, the notion of `not many' is quantified by $\eps$ instead of $1-\delta_d$, and the `smallest' allowed edge-length is $r_0$. Both $r_0$ and $\eps$ are parameters given by the user and not constants.} Formally, 
$$
\rho(I):=\inf\bset{j\ge r_0,\; \#\bset{J\in\mathcal E, J\subset Q_{j}(I)}\le \frac\eps2\cdot m_d\bb{Q_{j}(I)}},
$$
where $m_d$ denotes Lebesgue's measure on $\R^d$. For every $I\in \mathcal P$ if $I\subset \sbb{-N+\frac{N}{10^2d},N-\frac{N}{10^2d}}^d$, 
$$
\#\bset{I\in\mathcal E, J\subset Q_{\left\lfloor{\frac{N}{10^2d}}\right\rfloor}(I)}<\eps\cdot c_0\cdot N^d\le\frac\eps2\cdot m_d\bb{Q_{\left\lfloor{\frac{N}{10^2d}}\right\rfloor}(I)}
$$
as long as $c_0$ is chosen numerically small enough, depending on the dimension alone. Then for every such $I$, we know that $\rho(I)$ is finite. It could be infinite for $I$ close to the boundary of $Q$.
\begin{figure}[htp]
\begin{minipage}[c]{.50\textwidth}
For every basic cube $I$ we denote by $A(I,k)= Q_k(I)\setminus Q_{k-1}(I)$. Intuitively, this is the $k$th layer of basic cubes surrounding $I$ (see for example the alternating grey and black cubed annuli in Figure \ref{fig:layers} to the right). Let $Q_0:=\sbb{-N+\frac{2N}{10^2d},N-\frac{2N}{10^2d}}^d$. We will construct a monotone non-decreasing function  $M:\N\rightarrow\N$, so that for every $k$ fixed it satisfies
\begin{equation*}
\tag{M} \#\bset{I\in\mathcal P, k\in K_I}\ge\frac{11}{12}N^d, \label{eq:M}
\end{equation*}
\begin{equation*}
\tag{Q} \frac{M(k)}{M(k-M(k))}\le 4, \label{eq:quot}
\end{equation*}
where $K_I$ is a set defined for every basic cube $I\in\mathcal P$ with $I\subset Q_0$ by
\end{minipage}\hfill
\begin{minipage}[c]{.45\textwidth}
\centering
\includegraphics[scale=0.35]{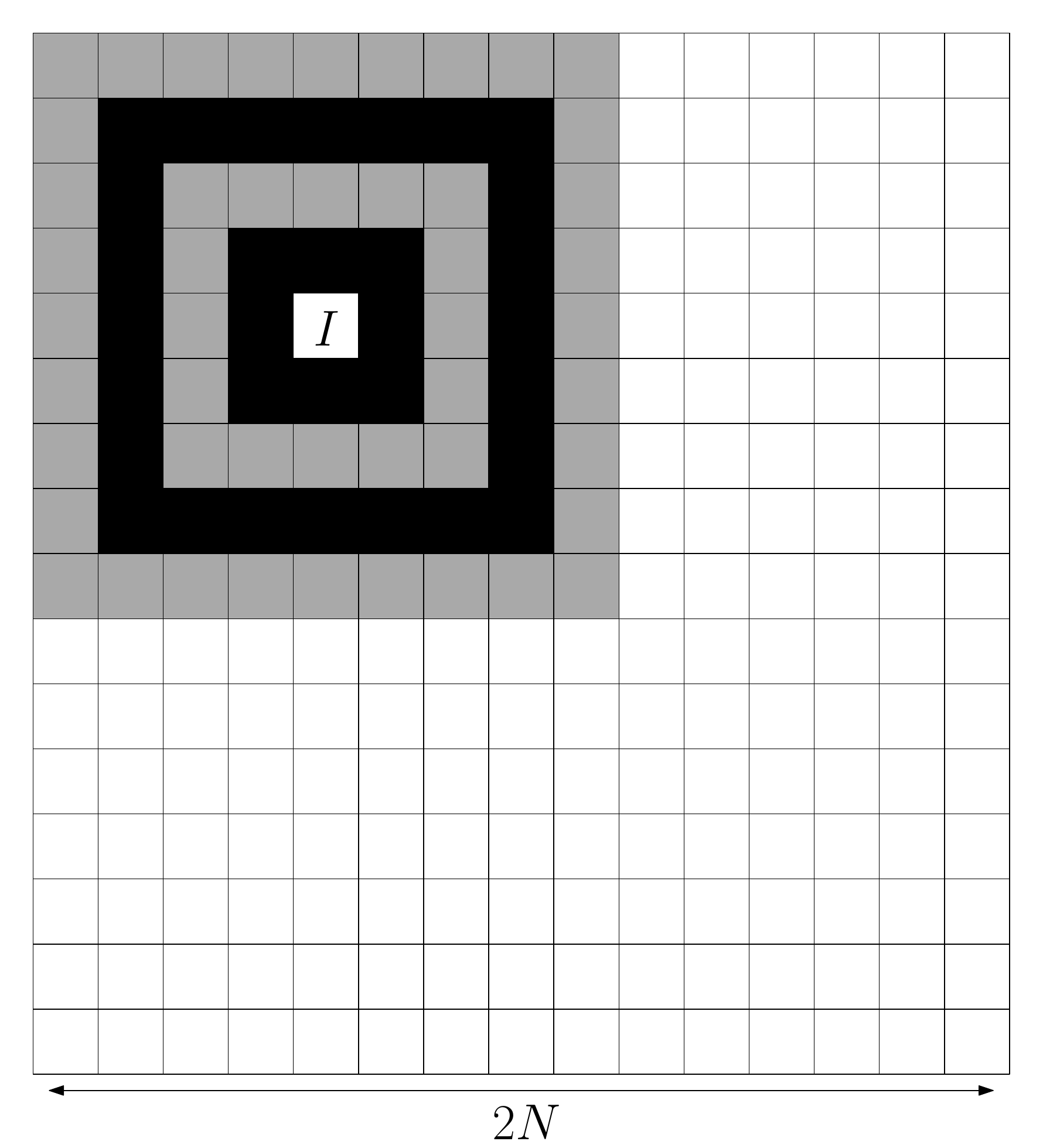}
\caption{In this figure one can see the layers about the cube $I$ which are coloured alternately.}
\label{fig:layers}
\end{minipage}
\end{figure}
$$
K_I:=\bset{k\in\bset{1,\cdots,\frac {N}{10^2d}},\; \forall x\in Q\cap A(I,k),\rho(x)\le\ M(k)}.
$$

Note that if $I\not\subset Q_0$, then it is possible that $\rho(x)=\infty$ for some $x\in Q\cap A(I,k)$ and $K_I$ is empty or small.\\
For now, we assume such a function exists.\\
For every cube $I\in \mathcal P$, define the monotone increasing sequence $\bset{\kappa_j}=\bset{\kappa_j(I)}$ in the following reversed recursive way:
\begin{eqnarray*}
\kappa_{N(I)}&:=&\max\bset{k\in K_I}\\
\kappa_j&:=&\max\bset{k\in K_I,\; k<\kappa_{j+1}-M(\kappa_{j+1})},
\end{eqnarray*}
where $N(I)$ is the eventual length of the sequence. We choose a cube $I$ so that $I\subset Q_0$, then in particular, $Q_{\left\lfloor{\frac {N}{10^2d}}\right\rfloor}(I)\subset Q$. Let $\gamma_j$ denote the boundary of $Q_j:=Q_{\kappa_j}(I)$, then by the way the set $K_I$ was defined, for every $x\in\gamma_j$,
$$
\rho(I(x))\le M(\kappa_j)\le dist\bb{\gamma_j,\gamma_{j+1}}\Rightarrow Q_{\rho(I(x))}(I(x))\subset Q_{j+1}.
$$
On the other hand, since $\kappa_{j-1}\le \kappa_j-M(\kappa_j)$, then 
$$
Q_{\rho(I(x))}(I(x))\cap Q_{j-1}=\emptyset.
$$
Overall, Property \ref{eq:contained} holds. Property \ref{eq:little_bad} holds by the way the function $\rho$ was defined. To conclude the proof, we need to bound from bellow $\#\bset{\kappa_j}$ for all but at most half of the basic cubes $I\in \mathcal P$.\\
Note that if the function $M$ is too big, the sequence $\bset{\kappa_j}$ will be a very short sequence, for most cubes. On the other hand, the smaller the function $M$ is, the harder it becomes to satisfy Property (\ref{eq:M}).

We shall conclude the proof in three steps:
\vspace{-1em}
\begin{enumerate}[label=Step \arabic*:]
\item
Show that at least $\frac{10}{11}N^d$ of the basic cubes $I\in\mathcal P$ satisfy
\begin{equation*}
\tag{$\S$}\#\bset{\kappa_j(I)}\ge \frac1{60}\sumit k 1 {\frac {N}{6d}}\frac 1{M(k)}.\label{eq:Mk_bound}
\end{equation*}
\item Define the function $M$, show it is non-decreasing, and prove that it satisfies conditions (\ref{eq:M}) and (\ref{eq:quot}).
\item Bound the sum $\sumit k 1 {\frac {N}{10^2d}}\frac 1{M(k)}$ from below.
\end{enumerate}
\subsubsection{Step 1:  Bounding $\sumit k 1 {\frac {N}{6d}}\frac 1{M(k)}$ from above}
\begin{claim}
For at least $\frac{10}{11}N^d$ of the basic cubes $I\in\mathcal P$, (\ref{eq:Mk_bound}) holds.
\end{claim}
\begin{proof}
For every basic cube $I\in\mathcal P$ let
$$
B(I):=\underset{k\in K_I}\sum\frac1{M(k)},
$$
and define the collection
$$
X:=\bset{I,\; B(I)\ge \frac 1{12}\sumit k 1 {\frac {N}{10^2d}}\frac1{M(k)}}.
$$
We will first show that every basic cube $I\in X$ satisfies (\ref{eq:Mk_bound}). Since $M$ is a monotone non-decreasing function with $M(1)\ge 1$, and by the way the sequence $\bset{\kappa_j}$ was defined, for every $I\in X$
\begin{eqnarray*}
\sumit k 1 {\frac {N}{10^2d}}\frac 1{M(k)}&\le& 12 B(I)=12\underset{k\in K_I}\sum\frac1{M(k)}=12\sumit j 2 {\#\bset{\kappa_j}} \underset{k\in K_I\atop \kappa_j-M(\kappa_j)\le k\le \kappa_j}\sum\frac 1{M(k)}\\
&\le&12\sumit j 2 {\#\bset{\kappa_j}} \sumit k{\kappa_j-M(\kappa_j)}{\kappa_j}\frac1{M(k)}\le 12\sumit j 1 {\#\bset{\kappa_j}}\frac{M(\kappa_j)+1}{M(\kappa_j-M(\kappa_j))}\le60\#\bset{\kappa_j},
\end{eqnarray*}
following Property (\ref{eq:quot}). To conclude the proof it is left to show that $\#X\ge\frac{10}{11}N^d$.\\
Following Property (\ref{eq:M}),
\begin{equation*}
\tag{$\Sigma$}\underset{I\in \mathcal P}\sum B(I)=\underset{I\in \mathcal P}\sum\;\underset{k\in K_I}\sum\frac1{M(k)}=\sumit k 1 {\frac {N}{10^2d}}\frac1{M(k)}\#\bset{I\in \mathcal P,\; k\in K_I}\ge\frac{11}{12}N^d\sumit k 1 {\frac {N}{10^2d}}\frac1{M(k)}.\label{eq:B(I)}
\end{equation*}
On the other hand, by the way the set $X$ was defined,
$$
B(I)\le \begin{cases}
		\frac 1{12}\sumit k 1 {\frac {N}{10^2d}}\frac1{M(k)}&, I \nin X\\
		\sumit k 1 {\frac {N}{10^2d}}\frac1{M(k)}&, I\in X
		\end{cases}.
$$
Then
\begin{eqnarray*}
\underset{I\in \mathcal P}\sum B(I)&=&\underset{I\in X}\sum B(I)+\underset{I\nin X}\sum B(I)\le \#X\sumit k 1{\frac N{10^2d}}\frac1{M(k)}+\bb{N^d-\#X}\frac1{12}\sumit k 1 {\frac {N}{10^2d}}\frac1{M(k)}\\
&=&\bb{\frac{11}{12}\#X+\frac1{12} N^d}\sumit k 1 {\frac{N}{10^2d}}\frac1{M(k)}\overset{\text{by (\ref{eq:B(I)})}}\le\frac{12}{11N^d}\bb{\frac{11}{12}\#X+\frac1{12}N^d} \underset{I\in \mathcal P}\sum B(I)=\bb{\frac{\#X}{N^d}+\frac1{11}}\underset{I\in \mathcal P}\sum B(I)\\
&\Rightarrow& \#X\ge\frac{10}{11}N^d,
\end{eqnarray*}
concluding the proof.
\end{proof}
\subsubsection{Step 2: constructing the function $M$:}
The definition of the function $M$ will heavily depend on a collection of sets, $\mathcal M$, with special properties. To define the collection $\mathcal M$ we will need the following definition:\vspace{-0.5em}
\begin{figure}[ht]
\begin{minipage}[b]{.44\textwidth}
Given $k$ an interval $I$ is called {\it a dyadic interval of order $k$} if there exists $j$ so that $I=\pm 2^{j\cdot k}+\left[0,2^k\right)$. A cube $\mathcal C$ is called {\it a dyadic cube of order $k$} if there are $d$ dyadic intervals of order $k$, $I_1,I_2,\cdots,I_d$ so that $\mathcal C=\prodit j 1 d I_j$. We say a cube $\mathcal C$ is {\it a dyadic cube} if there exists $k$ so that $\mathcal C$ is a dyadic cube of order $k$ (see Figure \ref{fig:dyadic} to the right). Note that a dyadic cube of order $k$ is a disjoint union of $2^{d\bb{k-j}}$ dyadic cubes of order $j$, and that every distinct dyadic cubes of the same order are disjoint. This means we can define a partial order on the collection of dyadic cubes.

\end{minipage}\hfill
\begin{minipage}[b]{.55\textwidth}
\centering
\includegraphics[scale=0.58]{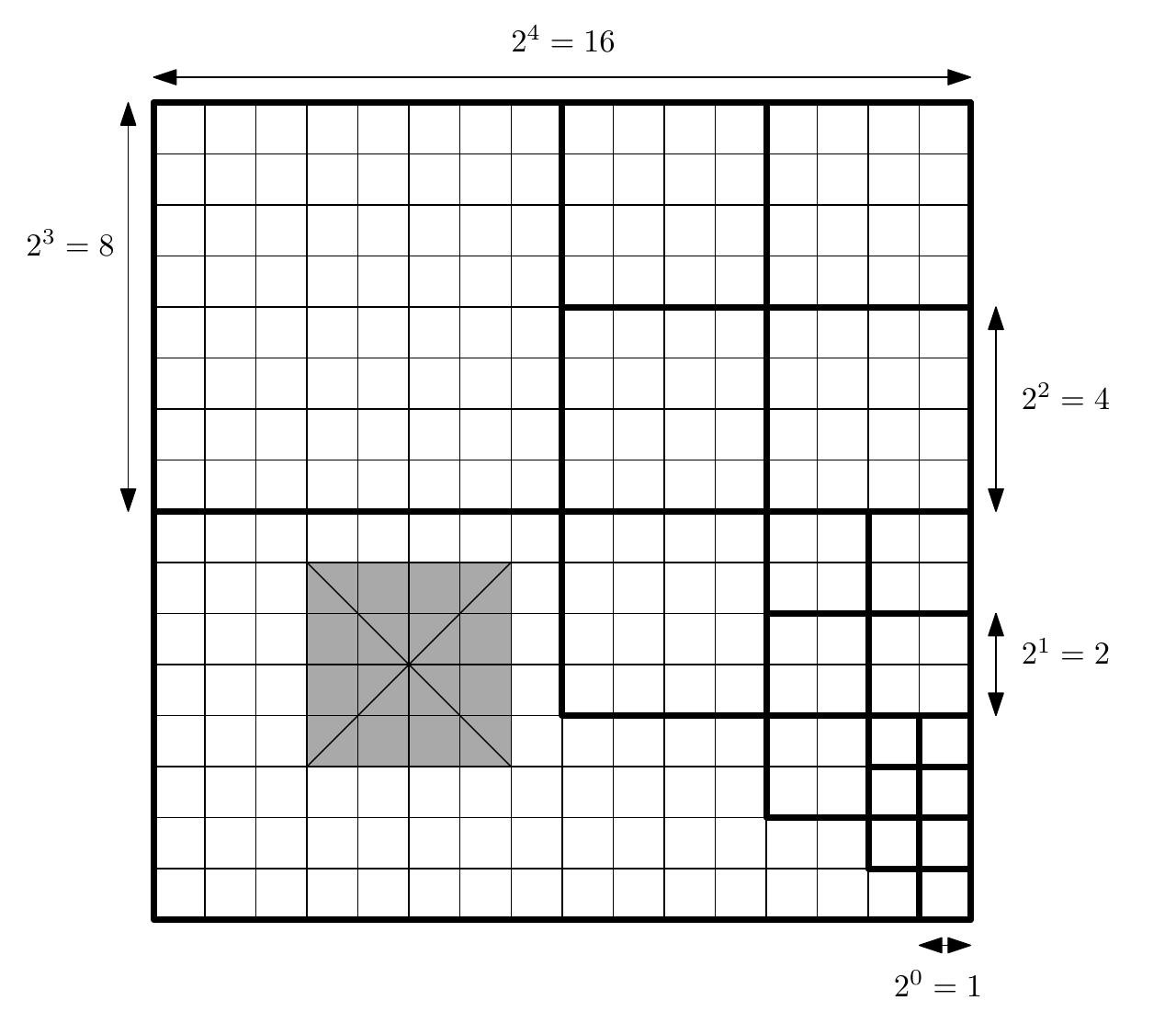}
\caption{This figure depicts dyadic cubes. The grey cube is not a dyadic cube even though it has edge-length 4.}
\label{fig:dyadic}
\end{minipage}
\end{figure}
\vspace{-1em}
 
We assume without loss of generality that $N$ is a dyadic integer. 
Note that $\rho$, defined on elements of $\mathcal P$, assumes a finite number of values, and one may order the cubes $I\in\mathcal P$ as $I_1,\cdots,I_{N^d}$ in an ascending $\rho$-order, i.e. so that $\rho(I_k)\ge\rho(I_{k+1})$. For every $I_k\in\mathcal P$ we find a dyadic cube $J_k$ so that:
\begin{enumerate}
\item $I_k\subset J_k$.
\item $2\rho(I_k)\le \ell(J_k)< 4\rho(I_k)$.
\end{enumerate}

We define the collection $\mathcal M$ to be the collection of maximal dyadic elements $\bset{J_k}$, that is
$$
J_k\in\mathcal M\iff I_k\not\subset \bunion i 1 {k-1}J_i,
$$
for if $I_k\subset J_i$ for some $i<k$, then since $\rho(I_i)\ge\rho(I_k)$, $J_k\subseteq J_i$. $\mathcal M$ forms a cover for $\sbb{-N+\frac N{10^2d},N-\frac N{10^2d}}^d$ and it is uniquely defined.\\
Let $n_\ell$ denote the number of elements in $\mathcal M$ with edge length $2^\ell$. We will define the function $M$ to be a step function, with steps
$$
s_m:=\min\bset{\bb{\frac{\alpha N^d}{\underset{\ell\ge m}\sum 2^{\ell} n_\ell}}^{\frac1{d-1}}+2^{m+2},\frac {N}{10^2d}},
$$
for some numerical constant $\alpha$ which we shall choose later. Note that this is a monotone non-decreasing sequence. We denote the index of the first element satisfying $s_m= \frac {N}{10^2d}$ by $\bar m+1$ and stop defining the sequence there. Let $m_0$ be a constant, which will depend on the dimension and on $r_0$, and will be determined later by Claim \ref{clm:large_sqrs_measure}, and define the function
$$
M(k):=	\begin{cases}
		2^{m_0}&, k\le s_{m_0}\\
		2^m&, s_{m-1}< k\le s_m,\;  m\in\bset{m_0+1,\cdots,\bar m+1}
		\end{cases}.
$$
Since the sequence $\bset{s_m}$ was monotone non-decreasing, the function $M$ is monotone non-decreasing as well.

\paragraph{The function $M$ satisfies Property (\ref{eq:quot}):} Fix $k$ and let $m$ be so that $s_{m-1}< k\le s_m$. Then, by monotonicity of the function $M$, if $m_0< k<\bar m+1$,
\begin{eqnarray*}
k-M(k)&\ge& s_{m-1}-M(s_m)=\bb{\frac{\alpha N^d}{\underset{\ell\ge m-1}\sum 2^{\ell} n_\ell}}^{\frac1{d-1}}+ 2^{m+1}-2^m\\
&=&\bb{\frac{\alpha N^d}{\underset{\ell\ge m-1}\sum 2^{\ell} n_\ell}}^{\frac1{d-1}}+2^{m}\ge \bb{\frac{\alpha N^d}{\underset{\ell\ge m-2}\sum 2^{\ell} n_\ell}}^{\frac1{d-1}}+2^{m}=s_{m-2}.
\end{eqnarray*}
In particular, by monotonicity of the function $M$,
$$
\frac{M(k)}{M(k-M(k))}\le\frac{M(s_m)}{M(s_{m-2})}=\frac{2^m}{2^{m-2}}=4.
$$

Though the definition of the steps of the function $M$ may seem forced and unnatural now, we will soon see it is in fact very natural.
\paragraph{The function $M$ satisfies Property (\ref{eq:M}):} We relay on two key claims. The first claim bounds the total number of elements in $\mathcal P$ with large $\rho$ by bounding the total measure of elements in $\mathcal M$ with large edge-length.
\begin{claim}\label{clm:large_sqrs_measure}
There exist a numerical constant $C_1$ which depend on the dimension $d$, and an index $m_0$ which depends on $r_0$ so that
$$
\underset{m\ge{m_0}}\sum 2^{m\cdot d} n_m\le \frac{C_1\cdot\#\mathcal E}\eps.
$$
\end{claim}
{This claim is a generalization of \cite[Claim 2.3]{MyJM2020}. While in the original proof $m_0$ only depended on the dimension, here it must also depend on the scale $r_0$. In addition, the constant $C_1$ depended on the dimension alone, as the upper bound on the density, $1-\delta_d$, was fixed and assumed to be bigger than $\frac12$. Here we must divide by $\eps$, which is now a parameter and may be as small as the user may wish. The proof is the same keeping track of these two parameters.}
\begin{proof}
Define the set
$$
K:=\underset{I\in\mathcal E}\bigcup I,
$$
and recall that the maximal function for cubes is defined by
$$
M_f(x):=\underset{C \text{ some cube}\atop\text{with } x\in C}\sup\frac1{m_d(C)}\underset C\int f(y)dm_d(y).
$$
Fix $x$ so that $\rho(I(x))>2r_0$, then by the way $\rho$ is defined
$$
\frac{m_d\bb{K\cap Q_{\frac{\rho(I(x))}2}(I(x))}}{m_d\bb{ Q_{\frac{\rho(I(x))}2}(I(x))}}>\frac{\frac\eps2\cdot m_d\bb{Q_{\frac{\rho(I(x))}2}(I(x))}}{m_d\bb{ Q_{\frac{\rho(I(x))}2}(I(x))}}=\frac\eps2,
$$
 and, by definition, the maximal function $M_{\indic{K}}$ satisfies
$$
M_{\indic{K}}(x)=\underset{C \text{ some cube}\atop\text{with } x\in C}\sup\frac1{m_d(C)}\underset C\int \indic{K}(y)dm_d(y)=\underset{C \text{ some cube}\atop\text{with } x\in C}\sup \frac{m_d(C\cap K)}{m_d(C)}\ge\frac{m_d\bb{K\cap Q_{\frac{\rho(I(x))}2}(I(x))}}{m_d\bb{ Q_{\frac{\rho(I(x))}2}(I(x))}}\ge\frac\eps2.
$$
We note that the same holds for every $y\in Q_{\frac{\rho(I(x))}2}(I(x))$. We conclude that $Q_{\frac{\rho(I(x))}2}(I(x))\subset\bset{M_{\indic{ K}}\ge\frac\eps2}$. Using the maximal function theorem for cubes,
\begin{equation*}
\tag{$\dagger$}\label{eq:max-func}
m_d\bb{\underset{\bset{x ,\;\rho(I(x))>2r_0}}\bigcup Q_{\frac{\rho(I(x))}2}(I(x))}\le m_d\bb{\bset{M_{\indic{K}}\ge\frac\eps2}}\lesssim \frac{m_d(K)}\eps= \frac{\#\mathcal E}\eps,
\end{equation*}
where the last equality is by the definition of the set $K$.

Let $J\in\mathcal M$ be so that $\ell(J)>8r_0$. Then there exists $x_J\in J$ with $\rho(I(x_J))\ge\frac{\ell(J)}4>2r_0$, implying that $J\subset Q_{5\rho(I(x))}(I(x))$.
We choose $m_0$ to be the first integer satisfying that $2^{m_0}>8r_0$. By using the inclusion above, since the elements in $\mathcal M$ are disjoint,
\begin{align*}
\sumit m {m_0}\infty 2^{d\cdot m} n_m&=\sumit m {m_0}\infty \bb{2^m}^d n_m=\underset{J\in\mathcal M\atop{\ell(J)\ge 2^{m_0}}}\sum m_d(J)=\underset{J\in\mathcal M\atop{\ell(J)\ge 2^{m_0}}}\sum m_d\bb{J\cap Q_{5\rho(I(x_J))}(I(x_J))}\\
&\overset{\text{disjointness}\atop\text{of }J}=m_d\bb{\underset{J\in\mathcal M\atop{\ell(J)\ge 2^{m_0}}}\bigcup\bb{J\cap Q_{5\rho(I(x_J))}(I(x_J))}}\overset{\text{inclusion}}\le m_d\bb{\underset{J\in\mathcal M\atop{\ell(J)\ge 2^{m_0}}}\bigcup Q_{5\rho(I(x))}(I(x))}\\
&\overset{(\star)}\le30^d m_d\bb{\underset{J\in\mathcal M\atop{\ell(J)\ge 2^{m_0}}}\bigcup Q_{\frac{\rho(I(x_J))}2}(I(x_J))}\le  30^dm_d\bb{\underset{\bset{x ,\;\rho(I(x))>2r_0}}\bigcup Q_{\frac{\rho(I(x))}2}(I(x))}\lesssim\frac{\#\mathcal E}\eps,
\end{align*}
the last inequality is by (\ref{eq:max-func}), concluding the proof. To show that $(\star)$ holds we look at two cases:\\

{\bf Case 1:} If the shrunken cubes are disjoint, then 
$$
10^d\cdot m_d(C_1\cup C_2)=10^d\cdot \bb{m_d(C_1)+m_d(C_2)}=m_d(10C_1)+m_d(10C_2)\ge m_d\bb{10C_1\cup 10 C_2}.
$$
Note that it is possible that $x_J$ is very close to the boundary of $J$ and therefore $Q_{\frac{\rho(I(x_J))}2}(I(x_J))\not\subset J$. In particular, it could be the case that for $J\neq J'$ we have $Q_{\frac{\rho(I(x_J))}2}(I(x_J))\cap Q_{\frac{\rho(I(x_{J'}))}2}(I(x_{J'}))\neq\emptyset$.

{\bf Case 2:} If the shrunken cubes intersect, and without loss of generality $\ell(C_1)\le\ell(C_2)$ (where $\ell(C)$ is the edge-length of the cube $C$), then $C_1\subset 3C_2$ implying that $10C_2\subset 30 C_2$ and therefore
$$
30^d\cdot m_d\bb{C_1\cup C_2}\ge30^d\cdot m_d\bb{C_2}=m_d(30 C_2)\ge m_d\bb{30 C_2\cup 10 C_1}\ge m_d\bb{10 C_2\cup 10 C_1}.
$$
The argument in Case 2 can be extended for any finite number of cubes, proving $(\star)$ holds.
\end{proof}

The second claim gives a softer condition to guarantee that Property (\ref{eq:M}) holds.
\begin{claim}\label{clm:prop2}
There exists a numerical constant $C_2$ so that for every $c_0$ small enough if $\alpha=\frac1{C_2}\bb{\frac1{12}-\frac1{50}}-c_0C_1>0$, then for all $m\ge m_0$, for every $k\le\bb{\frac{\alpha\cdot N^d}{\underset{\ell\ge m}\sum 2^{\ell} n_\ell}}^{\frac1{d-1}}$ at least $\frac{11}{12}$ of the basic cubes $I\in\mathcal P$ satisfy that $\left.\rho\right|_{A(I,k)\cap Q}\le 2^m$.
\end{claim}
It becomes clear now why the sequence $\bset{s_m}$ was defined in that particular way: Fix $k$ and let $m$ be so that $s_{m-1}<k\le s_m$, then by definition $M(k)=2^m$. Next,
$$
k\in K_I\iff \rho|_{A(I,k)\cap Q}\le M(k)=2^m.
$$
On the other hand,
$$
k\le s_{m}=\bb{\frac{\alpha\cdot N^d}{\underset{\ell\ge m}\sum 2^{\ell} n_\ell}}^{\frac1{d-1}}+2^{m+2}.
$$
Following the claim above, at least $\frac{11}{12}$ of the basic cubes $I\in\mathcal P$ satisfy that $\left.\rho\right|_{A(I,k)\cap Q}\le 2^m$, or, in other words, $\#\bset{I,\; k\in K_I}\ge\frac{11}{12}N^d$, and so the function $M$ satisfies Property (\ref{eq:M}).
\begin{proof}[Proof of Claim \ref{clm:prop2}]
For every $J\in\mathcal M$, for every basic cube $I\subset J$, $\rho(I)<\frac{\ell(J)}2$, by the way the collection $\mathcal M$ was defined. In particular, if $\ell(J)<2^m$, then for every basic cube $I\subset J$, $\rho(I)<2^m$ as well.
We conclude that for $I\subset Q_0$, if $A(I,k)$ only intersects $J\in\mathcal M$ with $\ell(J)<M(k)$, then every $x\in A(I,k)$ must satisfy $\rho(I(x))<M(k)$, as $\mathcal M$ forms a cover for $\sbb{-N+\frac{N}{10^2d},N-\frac{N}{10^2d}}^d$. Overall, we need to bound how many elements are in
$$
\bset{I\subset Q_0, \exists J\in\mathcal M, J\cap A(I,k)\neq\emptyset, \text{ and }\ell(J)>2^m}\subseteq\underset{J\in\mathcal M\atop \ell(J)>2^m}\bigcup\bset{I, A(I,k)\cap J\neq \emptyset}:=\underset{J\in\mathcal M\atop \ell(J)>2^m}\bigcup B(J,k).
$$
We will bound the number of elements in $B(J,k)$, i.e. given $J\in\mathcal M$ and $k\in\bset{1,\cdots,\frac N{10^2d}}$, how many elements $I\in \mathcal P$ satisfy that $J\cap A(I,k)\neq\emptyset$?

To bound this quantity we will look at two cases. The first and simpler case is when $\ell(J)>2k+1$. Since $\ell(J)>2k+1$, the cube $J$ must contain at least one outer vertex of the set $A(I,k)$. We begin by choosing some order on the outer vertices of $A(I,k)$. Then we can associate each basic cube $I\in B(J,k)$ with the location of the first outer vertex of $A(I,k)$ lying inside $J$ (see Figure \ref{fig:sub_big} bellow). Because $J$ is a union of basic cubes and so is $A(I,k)$, the number of sites in $J$, in which such a vertex could be found, is equal to the number of basic cubes in $J$. We can therefore bound from above the number of elements in $B(J,k)$, in this case, by the number of possible vertices in $A(I,k)$ times the number of basic cubes in $J$, that is by $2^dm_d(J)=2^d\cdot\ell(J)^d$.

The second case is when $\ell(J)\le 2k+1$. In this case, if $J\cap A(I,k)\neq\emptyset$, then the boundary of $J$ must intersect $A(I,k)$. In particular, there are two adjacent vertices of $J$ one in the set bounded by the outer boundary of $A(I,k)$ or in $A(I,k)$ and the other outside (or on) $A(I,k)$. We can identify every basic cube $I$ in $B(J,k)$ by indicating the basic cube on $A([0,1]^d,k)$ where the intersection of $A(I,k)$ and $J$ occurs, and the basic cube on the one dimensional facet connecting the two adjacent vertices of $J$ mentioned earlier  (see Figure \ref{fig:sub_small} bellow). If more than two vertices satisfy this, we choose an order on the collection of one dimensional facets and use the first one intersecting $A(I,k)$. We conclude that the number of elements in $B(J,k)$, in this case, is bounded by the number of basic cubes whose closure intersects a one dimensional facet of $J$ times the number of basic cubes in $A([0,1]^d,k)$, which is bounded by
$$
2^d\cdot\ell(J)\cdot 6d\cdot(2k+1)^{d-1}\lesssim \ell(J)\cdot k^{d-1}.
$$
\begin{figure}[!ht]
	 \center
	    \begin{subfigure}[b]{0.49\textwidth}
		{
		  \centering
	      	\includegraphics[width=0.8\linewidth]{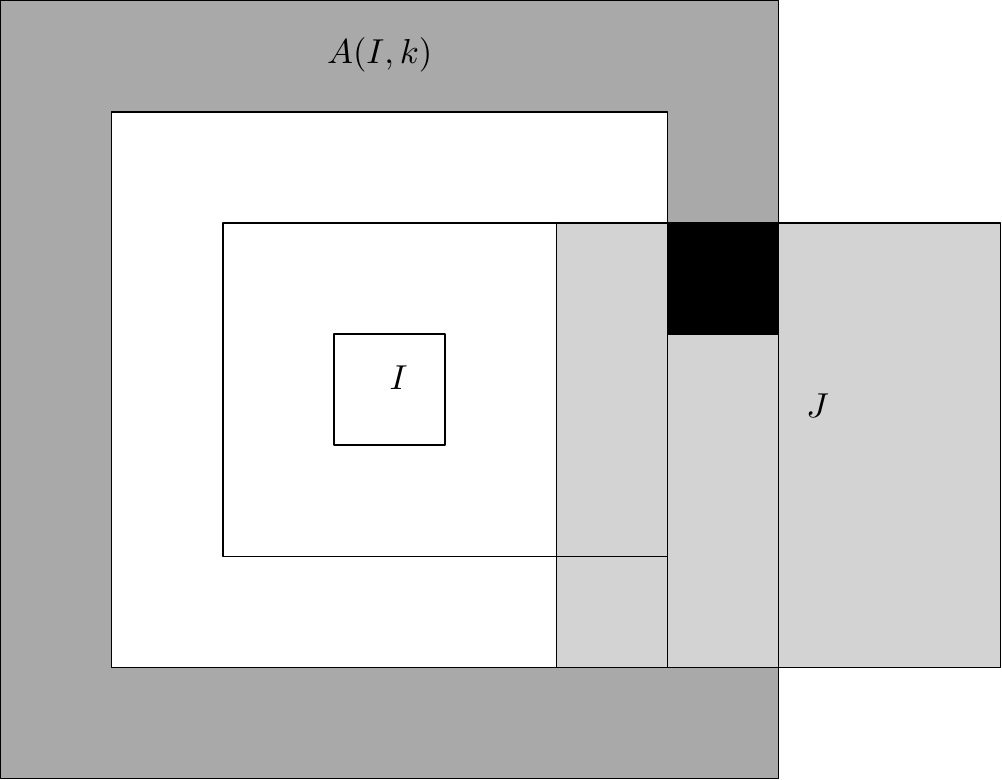}
		\caption{The case $\ell(J)\le2k+1$.}
		\label{fig:sub_small}
 		   }
	    \end{subfigure}
	    \begin{subfigure}[b]{0.49\textwidth}
		    {
		      \centering
	        	\includegraphics[width=0.8\linewidth]{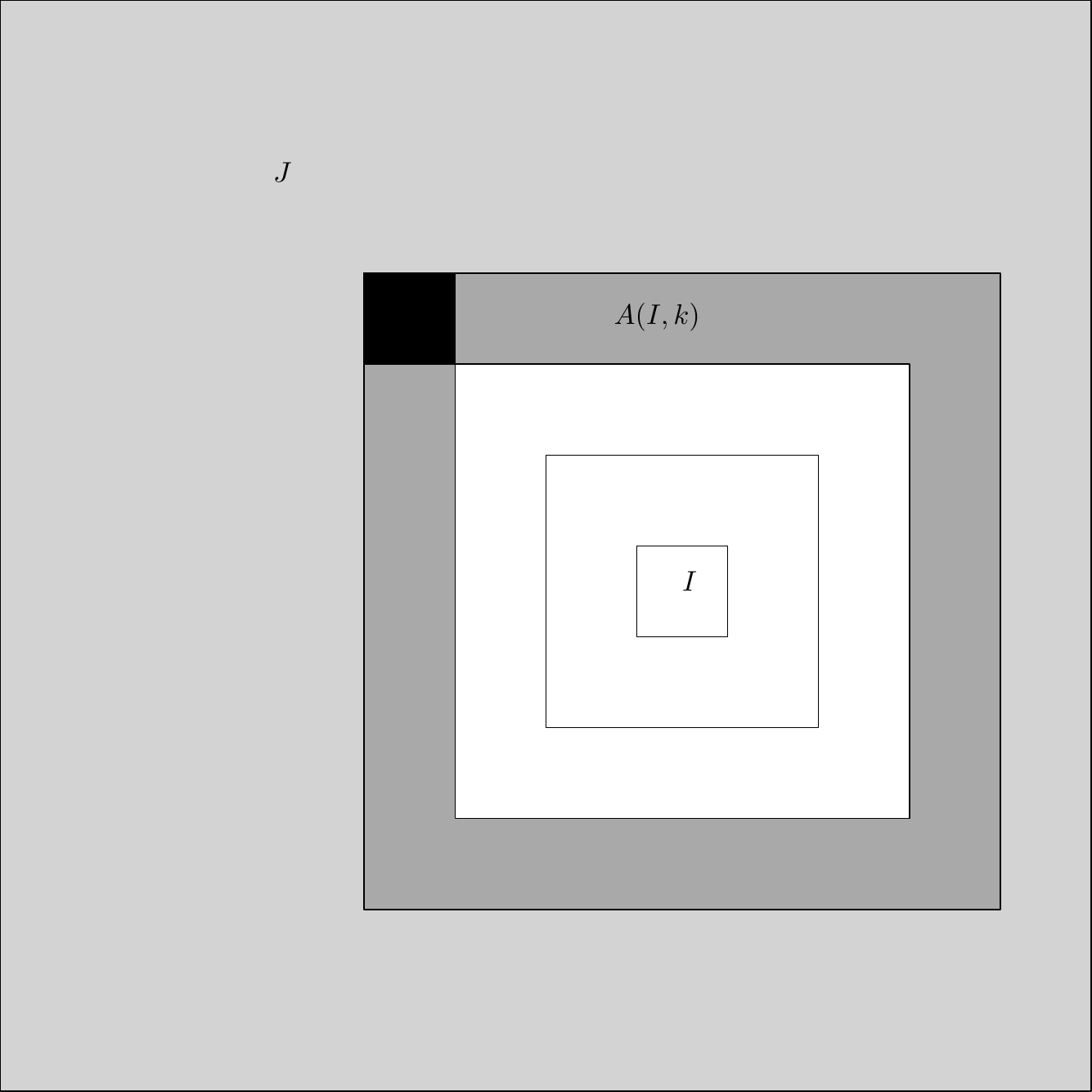}
		\caption{The case $\ell(J)>2k+1$. }
		 \label{fig:sub_big}
		    }
		    \end{subfigure}
	    \caption{This figure describes what happens in dimension $d=2$. In both figures, the black cube is the basic cube we choose to indicate $I\in B(J,k)$.}
	    \label{fig:clm2}
\end{figure}

Overall, we conclude that
\begin{align*}
&\#\bset{I\subset Q_0, \exists J\in\mathcal M, J\cap A(I,k)\neq\emptyset, \text{ and }\ell(J)>2^m}\le  \underset{J\in\mathcal M\atop \ell(J)>2^m}\sum \#B(J,k)\lesssim\underset{J\in\mathcal M\atop \ell(J)>2^m}\sum \bb{\ell(J)^d+k^{d-1}\ell(J)}\\
&\lesssim \underset{J\in\mathcal M\atop\ell(J)>2^m}\sum m_d(J)+k^{d-1}\underset{J\in\mathcal M\atop\ell(J)>2^m}\sum \ell(J)=\underset{\ell>m}\sum 2^{\ell\cdot d}n_\ell+k^{d-1}\underset{\ell>m}\sum 2^{\ell}n_\ell.
\end{align*}
Combining this with Claim \ref{clm:large_sqrs_measure}, we see that if $k^{d-1}\le\frac{\alpha N^d}{\underset{\ell\ge m}\sum 2^{\ell} n_\ell}$ and $\#\mathcal E\le \eps\cdot c_0N^d$, then there exists a constant $C_2$, which depends on the dimension alone, so that
\begin{align*}
&\#\bset{I\subset Q_0, \exists J\in\mathcal M, J\cap A(I,k)\neq\emptyset, \text{ and }\ell(J)>2^m}\le \cdots\le  C_2\underset{\ell> m}\sum 2^{\ell\cdot d}n_\ell+C_2k^{d-1}\underset{\ell\ge m}\sum 2^{\ell}n_\ell \\
&\le C_2\cdot \frac{C_1\#\mathcal E}\eps+C_2\alpha\cdot N^d\le N^d\bb{C_2C_1c_0+C_2 \alpha}=N^d\bb{\frac1{12}-\frac1{50}},
\end{align*}
as we assume that $\#\mathcal E\le \eps\cdot c_0\cdot N^d$, and by the way we defined $\alpha$. We therefore get that
\begin{align*}
&\#\bset{I\subset Q_0,\;\forall x\in Q\cap A(I,k) \text{ with }\rho(I(x))\le2^m}=\\
&=N^d\bb{1-\frac2{10^2d}}^d-\#\bset{I, \exists J\in\mathcal M, J\cap A(I,k)\neq\emptyset, \text{ and }\ell(J)>2^m}\\
&\ge N^d\bb{\bb{1-\frac2{10^2d}}^d-\bb{\frac1{12}-\frac1{50}}}>N^d\bb{1-\frac1{50}-\frac1{12}+\frac1{50}}=\frac{11}{12}N^d,
\end{align*}
concluding the proof.
\end{proof}
\subsubsection{Step 3:  Bounding $\sumit k 1 {\frac {N}{6d}}\frac 1{M(k)}$ from below}
\begin{claim}
$$
\sumit k 1 {\frac {N}{10^2d}}\frac 1{M(k)}\gtrsim \frac{N}{1+\bb{\frac{\#\mathcal E}{\eps\cdot N}}^{\frac1{d-1}}}\log^{\frac d{d-1}}\bb{2+\frac{\#\mathcal E}{\eps\cdot N}}.
$$
\end{claim}
\begin{proof}
If $s_{m_0}=\frac N{10^2d}$, then by the way the function $M$ is defined, for all $k$ we have $M(k)=2^{m_0}$, implying that,
$$
\sumit k 1 {\frac {N}{10^2d}} \frac1{M(k)}=\sumit k 1 {\frac {N}{10^2d}} \frac1{2^{m_0}}\sim N,
$$
where the constant depends on $m_0$, which depends on $r_0$ since this is how $m_0$ was defined. This yields a longer sequence than the one indicated in the theorem. Otherwise, assume that $s_{m_0}<\frac N{10^2d}$. Since the function $M$ was defined as a step function,
\begin{eqnarray*}
\sumit k 1 {\frac {N}{10^2d}} \frac1{M(k)}&=&\frac{s_{m_0}}{2^{m_0}}+\sumit m {m_0+1} {\bar m+1}2^{-m}\bb{s_m-s_{m-1}}=\frac{s_{m_0}}{2^{m_0}}+\sumit m {m_0+1} {\bar m+1}2^{-m}s_m-\sumit m {m_0+1} {\bar m+1}2^{-m}s_{m-1}\\
&=&\sumit m {m_0} {\bar m+1}2^{-m}s_m-\frac12\sumit m {m_0} {\bar m}2^{-m}s_m=\frac12\sumit m {m_0} {\bar m}2^{-m}s_m+2^{-\bar m-1}s_{\bar m+1}\sim N\cdot 2^{-\bar m}+\sumit m {m_0} {\bar m}\frac{s_m}{2^m}.
\end{eqnarray*}
To bound $\sumit m {m_0} {\bar m}\frac{s_m}{2^m}$ from bellow, we use Jensen's inequality with the convex function $g(t)=t^{-\frac1{d-1}}$:
\begin{eqnarray*}
\sumit m {m_0}{\bar m}\frac{s_m}{2^m}&=&\sumit m {m_0}{\bar m}g\bb{\frac{2^{m(d-1)}}{s_m^{d-1}}}=\bb{\bar m-m_0+1}\frac{\sumit m {m_0}{\bar m}g\bb{\frac{2^{m(d-1)}}{s_m^{d-1}}}}{\bar m-m_0+1}\\
&\ge&\bb{\bar m-m_0+1}g\bb{\frac{\sumit m {m_0}{\bar m}\frac{2^{m(d-1)}}{s_m^{d-1}}}{\bar m-m_0+1}}= \frac{\bb{\bar m-m_0+1}^{1+\frac1{d-1}}}{\bb{\sumit m {m_0} {\bar m}\frac{2^{m(d-1)}}{s_m^{d-1}}}^{\frac1{d-1}}}.
\end{eqnarray*}
To bound $\sumit m {m_0} {\bar m}\frac{2^{m(d-1)}}{s_m^{d-1}}$ from above, we note that by definition 
$$
s_m=\bb{\frac{\alpha\cdot N^d}{\underset{\ell\ge m}\sum 2^{\ell} n_\ell}}^{\frac1{d-1}}+2^{m+2}\ge \bb{\frac{\alpha\cdot N^d}{\underset{\ell\ge m}\sum 2^{\ell} n_\ell}}^{\frac1{d-1}}.
$$
This implies that
\begin{eqnarray*}
\sumit m{m_0}{\bar m}\frac{2^{m(d-1)}}{s_m^{d-1}}\le\frac1{\alpha N^d}\sumit m{m_0}{\bar m} 2^{m(d-1)}\underset{\ell\ge m}\sum 2^{\ell} n_\ell=\frac1{\alpha N^d}\underset{\ell\ge m_0}\sum 2^{\ell} n_\ell\sumit m{m_0}\ell 2^{m(d-1)}\sim \frac1{\alpha N^d}\underset{\ell\ge m_0}\sum 2^{d\cdot \ell} n_\ell\sim\frac{\#\mathcal E}{\eps\cdot N^d},
\end{eqnarray*}
following Claim \ref{clm:large_sqrs_measure}. Then, using the estimate above,
\begin{eqnarray*}
\sumit m {m_0}{\bar m}\frac{s_m}{2^m}&\ge& \frac{\bb{\bar m-m_0+1}^{1+\frac1{d-1}}}{\bb{\sumit m {m_0} {\bar m}\frac{2^{m(d-1)}}{s_m^{d-1}}}^{\frac1{d-1}}}\gtrsim \bar m^{1+\frac1{d-1}}\cdot\bb{\frac{\eps\cdot N^d}{\#\mathcal E}}^{\frac1{d-1}},
\end{eqnarray*}
implying that
\begin{eqnarray*}
\sumit k 1 {\frac {N}{6d}} \frac1{M(k)}\sim2^{-\bar m}\cdot N+\sumit m {m_0}{\bar m}\frac{s_m}{2^m}\gtrsim 2^{-\bar m}\cdot N+\bar m^{1+\frac1{d-1}}\cdot\bb{\frac{\eps\cdot N^d}{\#\mathcal E}}^{\frac1{d-1}}.
\end{eqnarray*}
We would like our lower bound to hold for every $\bar m$. To find the optimal inequality, we define the function
$$
\varphi(x):=2^{-x}+x^{1+\frac1{d-1}}\cdot\bb{\frac{N}{\eps\cdot\#\mathcal E}}^{\frac1{d-1}}.
$$
As we know nothing about $\bar m$, we will look for the absolute minimum of $\varphi$ in $\left[1,\frac N{10^2d}\right]$ and use whatever inequality this minimum satisfies:
$$
\varphi'(x)=-2^{-x}\log(2)+\bb{1+\frac1{d-1}}\bb{\frac{\eps\cdot N}{\#\mathcal E}}^{\frac1{d-1}}\cdot x^{\frac1{d-1}}=0\iff x^{\frac1{d-1}}2^x\sim\bb{\frac{\#\mathcal E}{\eps\cdot N}}^{\frac1{d-1}}
$$
and the later is a minimum point since $\varphi''\ge 0$. The lower bound this minimum produces is
$$
\frac{N}{1+\bb{\frac{\#\mathcal E}{\eps\cdot N}}^{\frac1{d-1}}}\log^{\frac d{d-1}}\bb{2+\frac{\#\mathcal E}{\eps\cdot N}},
$$
concluding the proof.
\end{proof}
\section{Examples}\label{sec:examples}
In this section we find a few examples of families of functions satisfying the requirements of the theorem, which are not the traditional subharmonic functions, treated already in  \cite{MyJM2020}.

\subsection{Examples related to Harnack's inequality}
In this subsection we will discuss two classes of examples. In both cases, the weighted mean value property arises from some form of Harnack's inequality, and in both we are dealing with solutions to special differential equations. Lastly, we will show in both examples that they belong to $\mathcal F(c,c,m_d)$ where $c$ is a constant which depends on the parameters of the differential equation.

Let $a:\R^d\rightarrow \R^{d\times d}$ be so that
\begin{enumerate}[label=$\bullet$]
\item For every $x\in\R^d$, $a(x)=\bset{a_{j,k}(x)}_{j,k=1}^d$ is a symmetric positive definite matrix, whose eigenvalues lie between $\lambda\inv$ and $\lambda$.
\item For every $1\le j,k\le d$ the function $x\mapsto a_{j,k}(x)$ is Lebesgue integrable.
\end{enumerate}

\subsubsection{Special elliptic differential equations}
In a paper from 1961, \cite{Moser1961}, J. Moser found a large class of functions which satisfy Harnack's inequality. Following Observation \ref{obs:Harnack}, all the functions in this class belong to $\mathcal F(c,c,m_d)$. Moser studied all the solutions of the uniformly elliptic differential equation induced by $a$ described above:
\begin{equation}\label{eq:Moser}
\underset{1\le j,k\le d}\sum \frac{\partial}{\partial x_j}\bb{a_{j,k}(x)\cdot\frac{\partial u}{\partial x_k}}=0.
\end{equation}
Let $D\subset\R^d$ be a domain. A function $u:\R^d\rightarrow\R$ which satisfies that
$$
\underset D\int\; \bb{u^2+\norm{\nabla u}{}^2}dm_d<\infty
$$
is called a {\it (weak) solution of (\ref{eq:Moser})} in a domain $D$ if for every continuously differentiable compactly supported function $\phi$,
$$
\underset D\int\; \binr{\nabla\phi}{a\nabla u}dm_d=\underset{1\le j,k\le d}\sum\; \underset D\int\; \frac{\partial \phi}{\partial x_j}\cdot a_{j,k}(x)\cdot\frac{\partial u}{\partial x_k}=0.
$$
\begin{thm}\cite[Theorem 1]{Moser1961}
If $u$ is a positive solution of (\ref{eq:Moser}) in a domain $D$ and $D'\subset D$ is a compact subset of $D$, then there exists a constant which depends on $\lambda, d, D, D'$ so that
$$
\underset{x\in D'}\min\; u(x)\le c\cdot \underset{x\in D'}\max\; u(x).
$$
\end{thm}
If $u$ is a solution of (\ref{eq:Moser}) in $\R^d$, then for every ball $B$, $u$ is a solution in a ball twice its size, and by rescaling the constant above does not depend on $B$ at all.

Following Observation \ref{obs:Harnack}, we see that if we fix the dimension and $\lambda$, then every weak solution of (\ref{eq:Moser}) belongs to $\mathcal F(c,c,m_d)$ for some constant $c$, which depends on the dimension, $d$, and on $\lambda$. In fact, since weak solutions of linear elliptic operators satisfy the maximum principle (see e.g. \cite[Theorem 8.1]{Gilbarger2001}), every such solution will belong to $\mathcal F(1,c,m_d)$.

\subsubsection{Special parabolic differential equations}
In another paper a few years later, Moser studied solutions of a `heat equation with weighted laplacian'. More accurately, he looked at non-negative solutions $u:\R^{d+1}\rightarrow\R_+$ of the parabolic differential equation
\begin{equation}\label{eq:Moser_heat}
\frac{\partial u}{\partial t}+\underset{1\le j,k\le d}\sum \frac{\partial}{\partial x_j}\bb{a_{j,k}(t,x)\cdot\frac{\partial u}{\partial x_k}}=0.
\end{equation}
Note that here $a=\bset{a_{j,k}(t,x)}_{j,k=1}^d:\R\times\R^d\rightarrow\R^{d\times d}$ satisfies the requirements described in the beginning of this subsection. A {\it weak solution of the differential equation (\ref{eq:Moser_heat}) in a domain $D$} is defined as a function for which the first derivatives, $u_t,u_{x_1},{\cdots}, u_{x_d}$ are square integrable in $D$ and which satisfies that for every continuously differentiable compactly supported function $\phi=\phi(t,x)$,
$$
\underset D\int\; \phi\cdot\frac{\partial u}{\partial t}+\binr{\nabla_x\phi}{a\nabla_x u}dm_d=\underset D\int\;\phi\cdot\frac{\partial u}{\partial t}dtdx+ \underset{1\le j,k\le d}\sum\; \underset D\int\; \frac{\partial \phi}{\partial x_j}\cdot a_{j,k}(x)\cdot\frac{\partial u}{\partial x_k}dtdx=0.
$$
As mentioned before, it is known that solutions of parabolic differential equations satisfy a weak maximum principle, and in particular would satisfy Condition (\ref{eq:max_pric}) with $A=1$. To show that a weighted mean value property holds we will need some kind of Harnack's principle. However, the Harnack type inequality Moser obtained in \cite{Moser1964} is not the standard one.
\begin{thm}\cite[Theorem 1]{Moser1964}\label{thm:Moser_heat}
Let $R=[0,T]\times[0,\rho]^d\subset\R^{d+1}$ be a rectangle. Given $0<\tau_1<\tau_2<\tau<T$ and $\rho'<\rho$ we define
$$
R_-:=[\tau_1,\tau_2]\times[0,\rho']^d,\; R_+:=[\tau,T]\times[0,\rho']^d
$$
(see Figure \ref{fig:Moser_a} bellow). There exists a constant $c$ which depends on $\lambda, d, \tau_1,\tau_2,\tau,T,\rho,\rho'$ so that for every weak positive solution of (\ref{eq:Moser_heat}) in $R$, 
$$
\underset{R_-}\max\; u\le c\cdot\underset{R_+}\min\; u.
$$
\end{thm}

\begin{obs}\label{obs:rec_to_ball}
Let $u$ be a weak positive solution of (\ref{eq:Moser_heat}). Then $u$ satisfies a weighted mean value property with $\mu=m_{d+1}$ and a constant which depends on the parameters of the equation.
\end{obs}
\begin{proof}
Fix $(t,x)\in \R\times \R^d$ and $r>0$ and define the rectangle
$$
R:=(x,t)+\sbb{-\frac r{\sqrt{d+1}},\frac r{\sqrt{d+1}}}^{d+1}.
$$
Note that $R\subset B((t,x),r)$. We define the parameters
$$
\tau_1:=-\frac r{4\sqrt{d+1}},\; \tau_2:=-\tau_1=\frac r{4\sqrt{d+1}}, \; \tau:=\frac r{2\sqrt{d+1}},\; \rho'=\frac {3r}{4\sqrt{d+1}},
$$
while, according to the way $R$ was defined, $\rho=T=\frac r{\sqrt{d+1}}$. Note that $(t,x)\in R_-$ and that $R_-, R_+\subset R\subset B((t,x),r)$ (see Figure \ref{fig:Moser_b} below). Following Theorem \ref{thm:Moser_heat}, there exists a constant $c$ which {may depend only} on $\lambda$ and the dimension $d$ so that
$$
\underset{R_-}\max\; u\le c\cdot\underset{R_+}\min\; u.
$$
We point out that the reason why it does not depend on the rest of the parameters is since the theorem is scale invariant and once we rescale everything to have $r=1$ we see that all the rest of the parameters are constants which depend on the dimension alone. Next, we note that $\frac{m_{d+1}(B((t,x),r))}{m_{d+1}(R_+)}$ is another constant which depends on the dimension alone. We see that by inclusion,
\begin{align*}
u(t,x)&\le \underset{R_-}\max\; u\le c\cdot\underset{R_+}\min\; u=\frac c{m_{d+1}(R_+)}\integrate {R_+}{}{\underset{R_+}\min\;  u(y)}{m_{d+1}}\le \frac c{m_{d+1}(R_+)}\integrate {R_+}{}{u(y)}{m_{d+1}}\\
&\le \frac{m_{d+1}(B((t,x),r))}{m_{d+1}(R_+)}\cdot \frac c{m_{d+1}(B((t,x),r))}\integrate {B((t,x),r)}{}{u(y)}{m_{d+1}}\le \frac {c'}{m_{d+1}(B((t,x),r))}\integrate {B((t,x),r)}{}{u(y)}{m_{d+1}},
\end{align*}
where $c'$ is a constant which depends on $\lambda$ and the dimension alone, concluding the proof.
\end{proof}

\begin{figure}[!ht]
	 \center
	    \begin{subfigure}[b]{0.49\textwidth}
		{
		  \centering
	      	\includegraphics[width=0.9\linewidth]{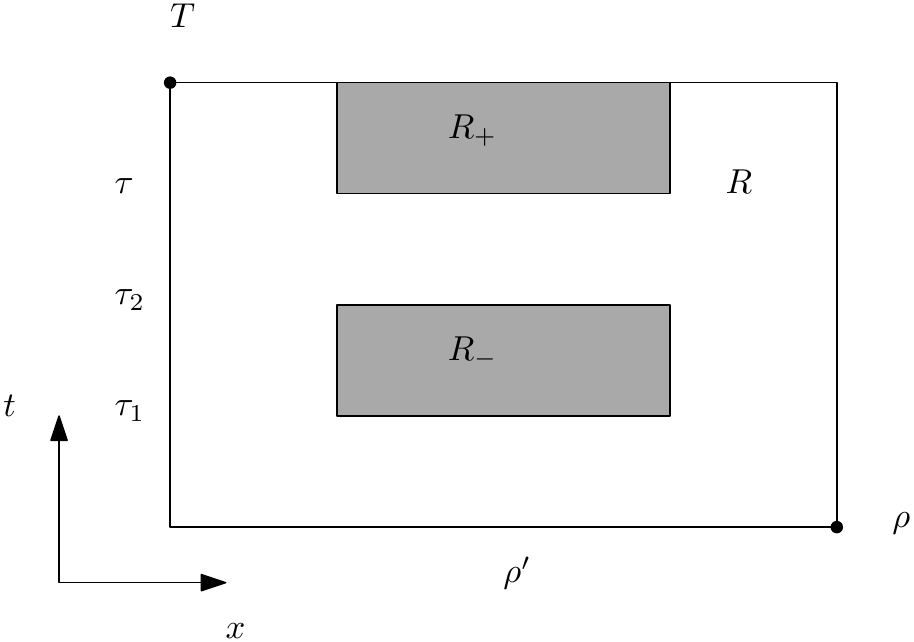}
		\caption{The rectangles $R,\; R_-,\; R_+$ {in} Theorem \ref{thm:Moser_heat}.}
		\label{fig:Moser_a}
 		   }
	    \end{subfigure}
	    \begin{subfigure}[b]{0.49\textwidth}
		    {
		      \centering
	        	\includegraphics[width=0.6\linewidth]{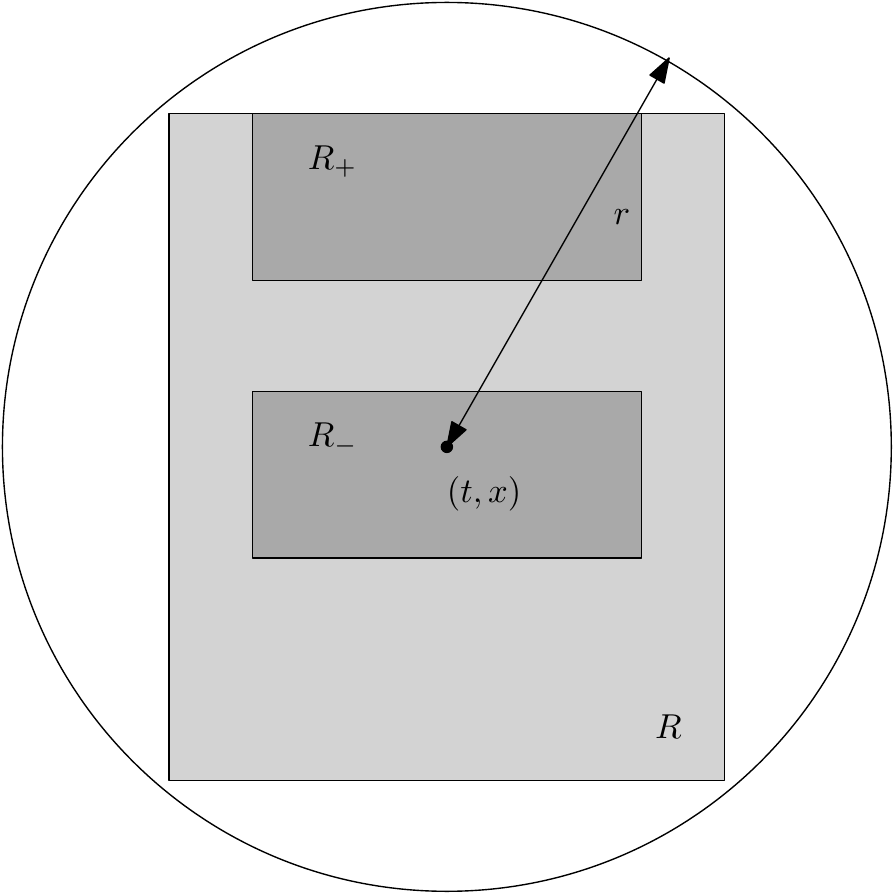}
		\caption{The ball $B((t,x),r)$ and the rectangles $R,\; R_-,\; R_+$ in the observation above. }
		 \label{fig:Moser_b}
		    }
		    \end{subfigure}
		    \caption{This figure describes the rectangles $R,R_-,R_+$ and how they relate with the ball $B((t,x),r)$ in the proof of Observation \ref{obs:rec_to_ball}.}
\end{figure}
Following the observation above, we see that if we fix the dimension and $\lambda$, then every weak solution of (\ref{eq:Moser_heat}) belongs to $\mathcal F(1,c,m_{d+1})$ for some constant $c$, which depends on the dimension and $\lambda$ alone.

\subsection{Functions satisfying a weighted average property}
The second class of examples we present have $A=B=1$ but the measure $\mu$ is not Lebesgue's measure.

In a paper from 1965, \cite{Bose1965}, A.K Bose investigated functions satisfying a weighted average property. More precisely, given a nonnegative locally bounded function $\omega:\R^d\rightarrow\R_+$ he looked at the collection of functions satisfying
\begin{equation}\label{eq:w-harmonic}
u(x)=\frac{\int_{B(x,r)} u(y)\cdot \omega(y)dm_d(y)}{\int_{B(x,r)} \omega(y)dm_d(y)}.
\end{equation}
Note that if $\omega=const\neq 0$ then the collection of functions satisfying (\ref{eq:w-harmonic}) is precisely the set of harmonic functions. However, if $\omega$ is not constant we get a different collection of functions.

He showed several interesting properties of functions satisfying (\ref{eq:w-harmonic}). First, a maximum principle holds:

\begin{thm}\cite[Theorem 6]{Bose1965}
Let $\Omega\subset\R^d$. If $u$ satisfies (\ref{eq:w-harmonic}) in $\Omega$, then $u$ cannot assume a maximum (minimum) at a point in $\Omega$ unless $u$ is constant in $\Omega$.
\end{thm}
This tells us that if $\omega$ is a measure satisfying Condition (\ref{eq:measure}), then every function satisfying (\ref{eq:w-harmonic}) belongs to $\mathcal F(1,1,\omega \cdot dm_d)$.

More interestingly, he showed that if $\omega$ is an eigenfunction of the laplacian, i.e., there exists $\lambda\in\R$ so that
$$
\Delta\omega+\lambda\omega=0,
$$
then $u\in C^2$ satisfies (\ref{eq:w-harmonic}) if and only if it is a solution of the equation
$$
\omega\cdot\Delta u+2\binr{\nabla u}{\nabla w}=0.
$$
For more information see \cite[Theorem 5]{Bose1965}. In this case, the characterization of functions satisfying an $\omega$-weighted mean value theorem, as described in (\ref{eq:w-harmonic}), can give rise to defining $\omega$-subharmonic functions, as sub-solutions of the equation above. These will also belong to $\mathcal F(1,1,\omega\cdot dm_d)$ and Theorem \ref{thm:quasi-lin-subsol} could also be applied to them.

If $\omega$ is not an eigenfunction of the laplacian, then every function satisfying (\ref{eq:w-harmonic}) will {still} be a solution of the equation above, but there are solutions of the equation above, which do not satisfy (\ref{eq:w-harmonic}). For more information see \cite[Theorem 4]{Bose1965} and the remark that follows the proof.


\subsection{Sub-solutions of quasilinear elliptic equations}
The last class of examples we would like to mention, is where $B>1$ and the measure $\mu$ is not necessarily Lebesgue's measure.

Let $A_p{,\; 1< p<\infty, }$ denote the Muckenhoupt class which consists of all nonnegative locally integrable functions $\omega$ in $\R^d$ satisfying
$$
\underset{B\text{ ball}}\sup \bb{\frac 1{m_d(B)}\int_B\omega(x) dx}\bb{\frac 1{m_d(B)}\int_B\omega(x)^{\frac1{1-p}} dx}^{p-1}:=c_{p,\omega}<\infty.
$$
If $\omega\in A_p$ then it is also $p$-admissible with the same $p$. (See \cite[p.10]{Heinonen}).

Let $A:\R^d\times\R^d\rightarrow\R^d$ be a mapping satisfying that
\begin{enumerate}
\item
\begin{enumerate}
\item The mapping $x\mapsto A(x,\xi)$ is measurable for all $\xi\in\R^d$.
\item The mapping $\xi\mapsto A(x,\xi)$ is continuous for a.e $x\in\R^d$.
\end{enumerate}
\item There exists $\alpha>0$ so that for all $\xi\in\R^d$ and a.e. $x\in\R^d$ we have
$$
\binr{A(x,\xi)}\xi\ge\alpha\cdot\omega(x)\abs\xi^p.
$$
\item There exists $\beta>0$ so that for all $\xi\in\R^d$ and a.e. $x\in\R^d$ we have
$$
\abs{A(x,\xi)}\le\beta\cdot\omega(x)\abs\xi^{p-1}.
$$
\item For every $\xi_1\neq\xi_2$, $\binr{A(x,\xi_1)-A(x,\xi_2)}{\xi_1-\xi_2}>0$.
\item For every $\lambda\in\R, \lambda\neq 0$, $A(x,\lambda\xi)=\lambda\abs\lambda^{p-2}A(x,\xi)$.
\end{enumerate}
For a mapping satisfying properties 1-5 above we look at the quasilinear elliptic equation
\begin{equation*}
\tag{$\cancer$} -div\bb{A(x,\nabla u)}=0. \label{eq:quasilinear}
\end{equation*}

Let $\Omega\subset\R^d$ be a bounded domain. A function $u$ is called {\it $A$-harmonic in $\Omega$} if it is a continuous weak solution for (\ref{eq:quasilinear}) in $\Omega$. A function $u$ is called {\it $A$-subharmonic in $\Omega$} if 
\begin{enumerate}[label=(\alph*)]
\item It is upper semicontinuous.
\item In every component of $\Omega$, $u\not\equiv -\infty$.
\item For every $D\Subset\Omega$ open, for every $A$-harmonic function $v$ which is continuous on $\overline D$, if $u\le v$ on $\partial D$, then $u\le v$ in $D$.
\end{enumerate}
{This class contains all $p$-subharmonic functions, when $\mathcal A(x,y)=\abs y^{p-2}\cdot y$}.

While subharmonic functions and $A$-subharmonic functions share some properties, there are some very elementary properties of subharmonic functions that require modifications to extend for the more general setup, and others that do not extend at all. For example, if $u,v$ are both $A$-subharmonic it is NOT necessarily the case that $u+v$ is $A$-subharmonic. We list here a partial list of interesting properties:
\begin{enumerate}
\item If $\lambda\ge 0$ and $\tau\in\R$ then for every $A$-subharmonic function $\lambda\cdot u+\tau$ is also $A$-subharmonic.
\item If $u,v$ are both $A$-subharmonic then $\max\bset{u,v}$ is $A$-subharmonic.
\item {\bf A maximum principle:} A non-constant $A$-subharmonic function $u$ cannot attain its supremum inside the domain $\Omega$.
\item Every $A$-subharmonic function is a sub-solution of Equation (\ref{eq:quasilinear}).
\item {\bf A weighted mean-value property:} Let $u$ be a non-negative $A$-subharmonic function in a ball $B$, then
$$
u(x)=ess\limitsup y x u(y)\le C\cdot\frac1{\mu(B)}\underset B\int u d\mu,
$$
where $\mu$ is the measure associated with the weight $\omega$, i.e. $\mu(E):=\underset E\int \omega(x)dx$, and $C$ is a constant which depends on $\alpha,\;\beta, \;p,\; d$ and the constant that controls the Poincare-type inequality the measure $\mu$ satisfies (see \cite[Lemma 3.44]{Heinonen}).
\end{enumerate}

To see that $\mu$ satisfies Condition (\ref{eq:measure}) we refer the reader to the following lemma:
\begin{lem}\label{lem:doubling}\cite[Lemma 15.8, p. 299]{Heinonen}
If $\omega\in A_p$, then there exists $q\in(0,1)$ and $c>1$ both depending only on $n,p,\; c_{p,\omega}$ so that for every ball $B$ and every measurable set $E\subset B$
$$
\frac{\mu(E)}{\mu(B)}\le c\bb{\frac{m_d(E)}{m_d(B)}}^q.
$$
\end{lem}

\begin{obs}
For every $A$-subharmonic function, $u$ there {exist} constants $c,B>1, q\in\bb{0,1}$ so that $u\in\mathcal F(1,B,\mu)$ for some measure $\mu\in\mathcal M_\psi(\R^d)$ where $\psi(t):=c\cdot t^q$.
\end{obs}
\begin{proof}
Following Property 3, every $A$-subharmonic function satisfies the maximum principle , and therefore Condition (\ref{eq:max_pric}) holds with the constant $A=1$. Next, following Lemma \ref{lem:doubling}, the measure $\mu$ satisfies Condition (\ref{eq:measure}) with the function $\psi(t):=c\cdot t^q$ for the constant $c>1,\; q\in\bb{0,1}$ in the lemma above. Lastly, following the mean-value property, Property 5 listed above, the function $u$ satisfies Condition (\ref{eq:mean_val}) with the constant $B=C$ where $C$ is the constant appearing in Property 5, which depends on $\alpha,\beta,\; p,\; d$.
This concludes the proof of the observation.
\end{proof}

\nocite{*}
\bibliographystyle{plain}
\bibliography{}

\bigskip
\noindent A.G.:
Department of Mathematics, Northwestern, Illinois, USA.
\newline{\tt adiglucksam@gmail.com} 

\end{document}